\documentclass[11pt,reqno]{amsart}

\usepackage{amsmath}
\usepackage{amssymb}
\usepackage{amsthm}
\usepackage{color}
\usepackage{eucal}

\DeclareSymbolFont{rsfscript}{OMS}{rsfs}{m}{n}
\DeclareSymbolFontAlphabet{\mathrsfs}{rsfscript}

\def\softd{{\leavevmode\setbox1=\hbox{d}\hbox
            to 1.15\wd1{d\kern-0.2ex{\char039}\hss}}}    
\def\softl{l\kern-0.3ex\raise0.1ex\hbox{'}\kern-0.3ex}   
\newtheorem{Thm}{Theorem}
\newtheorem{Prop}[Thm]{Proposition}
\newtheorem{Lemma}[Thm]{Lemma}
\newtheorem{Cor}[Thm]{Corollary}

\theoremstyle{remark}

\def\om{\omega}

\def\circledm{\protect\mathbin{\hbox
    {\protect$\bigcirc$\rlap{\kern-8.2pt\raise0pt\hbox
    {\protect$\mathtt{m}$}}}}}        
\def\malc{\circledm}            
\def\smallcircledm{\protect\mathbin{\hbox
    {\protect$\bigcirc$\rlap{\kern-6.9pt\raise0pt\hbox
    {\protect$\mathtt{m}$}}}}}
\def\malcab{\smallcircledm}

\title[]{On the power pseudovariety $\mathbf{PCS}$}

\author{K.~Auinger}
\address{Fakult\"at f\"ur Mathematik, Universit\"at Wien,
Nordbergstrasse 15, A-1090 Wien, Austria}
\email{karl.auinger@univie.ac.at}

\begin{document}

\begin{abstract}
The pseudovariety $\mathbf{PCS}$ {which is} generated by all
power semigroups of {finite} completely simple semigroups
is characterized in various ways. For example, the equalities
     \[\mathbf{PCS}=\mathbf{J}\malcab\mathbf{CS}
     =\mathbf{BG}\malcab\mathbf{RB}\]
are established. {This resolves a problem raised by Ka\softd ourek and leads to several transparent algorithms for deciding
membership in  $\mathbf{PCS}$}.
\end{abstract}
\maketitle    

In the study of pseudovarieties of finite semigroups,
the \emph{power operator} $\mathbf P$
that assigns to {every} pseudovariety  $\mathbf{V}$ the
pseudovariety $\mathbf{PV}$ generated by all power semigroups
of {the} members of $\mathbf{V}$ has attracted considerable
attention{ --- see Chapter~11 of the monograph}~\cite
{almeidauniv} and the survey article \cite{almeidapower} by Almeida.
For example, the recognition of the  pseudovariety $\mathbf{PG}$
generated by all power groups (that is, power semigroups of groups)
 as the pseudovariety {$\mathbf{BG}$}
of all {block groups}, often stated as the famous equality
\begin{equation}\label{pg=bg}\mathbf{PG}=\mathbf{BG},\tag{$\sharp$}\end{equation}
is a celebrated result in finite semigroup theory.

The proof of the equality (\ref{pg=bg}) was originally obtained by
Henckell and Rhodes~\cite{henckrhodes} as a~consequence of a~deep
result of Ash~\cite{ashinevit}.  The reader interested in the
historical background, {in further developments} and
in generalizations of the above-mentioned equality is referred to
{Section 4.17 of the monograph}~\cite{qtheory}  {by Rhodes and Steinberg} and {to}
the introduction of Ka\softd ourek's paper~\cite{kadpcs}. {The equality (\ref{pg=bg}) in particular implies the decidability of membership in  $\mathbf{PG}$ which does not follow from the definition of the power operator. Indeed, the operator $\mathbf{P}$ does not preserve decidability of membership \cite{auinsteinbpartial}.}

It seems to be natural to replace ``groups'' with ``completely
simple semigroups'' and to ask for a suitable description
of the power pseudovariety $\mathbf{PCS}$. This question occurs in the list of open problems in~\cite{almeidauniv}.     
The first solution to this was found by Steinberg
\cite{steinbpcs} who proved the equality
     \[\mathbf{PCS}=\mathbf{BG}*\mathbf{RZ}\]     
which  implied the decidability of membership
in $\mathbf{PCS}$. However, neither a transparent structural description of the members of $\mathbf{PCS}$ nor an efficient algorithm for the membership problem in $\mathbf{PCS}$ were presented. The latter fact was considered by Ka\softd ourek  to be not completely satisfactory; he therefore proposed further research on that topic: in \cite{kadpcs} he
studied  finite semigroups~$S$ all of whose subsemigroups of the form $aSb$ (for $a,b\in S$) are block
groups.  It is easy to see that these  exactly comprise
the Ma\softl cev product $\mathbf{BG}\malc \mathbf{RB}$, see Corollary \ref{agbg} below.  Ka\softd ourek
proved that this class can be represented also as the Ma\softl cev product $\mathbf{J}\malc\mathbf{CS}$. Moreover, he related   $\mathbf{PCS}$ with the latter pseudovariety
by proving the inclusion $\mathbf{PCS}\subseteq\mathbf{J}\malc\mathbf{CS}$. The question whether that inclusion is
proper or not has been left open. It is the intention of the present
paper to resolve this problem. In fact, we shall prove
more, namely {we shall} establish the equalities
     \[\mathbf{PCS}=\mathbf{J}\malc\mathbf{CS}=
     \mathbf{BG}\malc\mathbf{RB}=\mathbf{ER}\malc\mathbf{RZ}\cap
     \mathbf{EL}\malc\mathbf{LZ}=\mathbf{BG}*\mathbf{RZ}.\]
The paper is organized as follows. In Section~\ref{preliminaries} we shall mention all preliminaries needed while
in Section~\ref{mainresult} we shall formulate and prove the
main result. This will be done semantically as well as syntactically;
the members of  $\mathbf{PCS}$
will thus be characterized in various ways.

\section{Preliminaries}
\label{preliminaries}

\subsection{Definitions, notation.} The reader is assumed to be familiar with central facts of semigroup theory, in particular with the theory of pseudovarieties of finite semigroups, including pseudoidentities and Tilson's derived semigroupoid theorem;  sources are  the monographs \cite{almeidauniv,qtheory} as well as Tilson's seminal paper \cite{tilson}. All semigroups considered in this paper are finite.

We start by introducing some notation. Given a~semigroup $S$ and {an
element} ${a}\in S$, then $L({a})=S{a}\cup\{{a}\}$
is the principal left ideal of~$S$ generated  by~$a$,
and likewise, $R({a})={a}S\cup\{{a}\}$ is the principal
right ideal of~$S$ generated  by $a$. Further, we
denote by  $L(a)^\rho$ the semigroup of all inner right translations (acting on the right) $\rho_s:L(a)\to L(a)$, $x\mapsto xs$ for $s\in L(a)$; then $\rho^{L(a)}:L(a)\twoheadrightarrow L(a)^\rho$, $s\mapsto \rho_s$ is a (not necessarily injective) homomorphism, called the \emph{right regular representation} of $L(a)$. Dually, ${}^\lambda\!R(a)$ denotes the semigroup of all inner left translations (acting on the left) $\lambda_s:R(a)\to R(a)$, $x\mapsto sx$ for $s\in R(a)$; then $\lambda^{R(a)}:R(a)\twoheadrightarrow {}^\lambda\!R(a)$, $s\mapsto \lambda_s$ is a (not necessarily injective) homomorphism, called the \emph{left regular representation} of $R(a)$.

Let $S$ and $T$ be {two} semigroups, both
generated by {the same} set~$A$.  Then the
subsemigroup of {the direct product} $S\times T$ generated by the set $\{(a,a)\mid a\in A\}$, \textbf{considered as a relation} from $S$ to $T$, is the \emph{canonical
relational morphism  $S\to T$ with respect to~$A$}.

Pseudovarieties of semigroups are usually denoted by bold-face letters $\mathbf{V}, \mathbf{W}$, etc.
Let $\mathbf{RZ}$, $\mathbf{LZ}$, $\mathbf{RB}$, 
{$\mathbf{G}$,} $\mathbf{CS}$, and $\mathbf{J}$ be, respectively,
the pseudovarieties of all  right zero semigroups, left zero
semigroups, rectangular bands,
{groups,} completely simple
semigroups, and $\mathrsfs{J}$-trivial semigroups.  Of central importance
will be the pseudovariety {$\mathbf{BG}$} of all
\emph{block groups}. It  consists of all semigroups all of
whose regular elements have a {unique} inverse.  {Equivalently, $\mathbf{BG}$ is comprised of}
all semigroups which do not contain
non-trivial right zero  and  left zero subsemigroups. It is well known  {that} $\mathbf{BG}$ is defined by the single pseudoidentity
$(x^\om y^\om)^\om=(y^\om x^\om)^\om$. In addition, we shall need the pseudovariety $\mathbf{ER}$
 {which consists} of all semigroups whose idempotent generated subsemigroups  are $\mathrsfs{R}$-trivial.
A semigroup  then belongs to $\mathbf{ER}$ if and only if
it does not contain a non-trivial right zero
subsemigroup. It is well known that
$\mathbf{ER}$ is defined by the {single} pseudoidentity
$(x^\om y^\om)^\om =(x^\om y^\om)^\om x^\om$. At last, we shall
need yet the pseudovariety $\mathbf{EL}$ {which} is defined
dually, that is, it consists of all
semigroups containing no non-trivial left zero subsemigroup. Likewise,  $\mathbf{EL}$ is defined by the single
pseudoidentity $(y^\om x^\om)^\om=x^\om(y^\om x^\om)^\om$.
 {For the latter three pseudovarieties  the equality $\mathbf{BG}=\mathbf{ER}\cap\mathbf{EL}$ holds.}
All pseudovarieties considered in this paper are pseudovarieties of \textbf{semigroups}.

For a semigroup $S$, let $\mathfrak{P}(S)$ be the set of all non-empty subsets of $S$; endowed with set-wise multiplication, the set $\mathfrak{P}(S)$ itself becomes a semigroup. For a homomorphism $\alpha:S\to T$ between semigroups $S$ and $T$, the induced mapping $\mathfrak{P}(S)\to \mathfrak{P}(T)$, $X\mapsto \{x\alpha\mid x\in X\}$ is a homomorphism, the \emph{induced homomorphism}. For a pseudovariety  $\mathbf{V}$ denote by $\mathbf{P}'\mathbf{V}$ the pseudovariety generated by all power semigroups $\mathfrak{P}(S)$ with $S\in \mathbf{V}$. Likewise, the full powerset $\mathfrak{P}(S)\cup\{\emptyset\}$ of a semigroup $S$ is also a semigroup under set-wise multiplication. For a pseudovariety $\mathbf{V}$ let $\mathbf{PV}$ be the pseudovariety generated by all full power semigroups $\mathfrak{P}(S)\cup\{\emptyset\}$ with $S\in \mathbf{V}$.
 It is well known  that $\mathbf{P'V}$ and $\mathbf{PV}$ coincide if and only if $\mathbf{V}$  contains a nontrivial monoid, see Lemma 5.1  in \cite{almeidapower}. In all cases considered in this paper, the equality $\mathbf{P}'\mathbf{V}=\mathbf{PV}$ holds.

 Finally, a semigroup $S$ equipped with a partial order $\le$ is an \emph{ordered semigroup} if $a\le b$ implies $ac\le bc$ and $ca\le cb$ for all $a,b,c\in S$. It is well known and easy to see that an ordered monoid $M$ satisfying $a\le 1$ for all $a\in M$  is $\mathrsfs{J}$-trivial \cite{pinorder}.

\subsection{Pseudovarieties of the  forms $\mathbf{V}\malc\mathbf{RZ}$,
$\mathbf{V}\malc\mathbf{LZ}$,  $\mathbf{V}\malc\mathbf{RB}$, and $\mathbf{V}*\mathbf{RZ}$}\label{section12}
The results in this subsection are likely
to be known  {by} experts, but the author is not aware of a proper reference.
For arbitrary pseudovarieties $\mathbf{V}$ and $\mathbf{W}$, the \emph{{Ma\softl cev} product} $\mathbf{V}\malc
\mathbf{W}$ consists of all semigroups~$S$ for which there exists
a~semigroup $T\in\mathbf{W}$ and a~relational morphism $\phi:S\to T$
such that, for each idempotent $e\in T$, the inverse image of~$e$
under~$\phi$, {that is,} $e\phi^{-1}=\{s\in S\mid (s,e)\in\phi\}$
(which is a~subsemigroup of~$S$) belongs to $\mathbf{V}$. Suppose that~$\mathbf{W}$
is locally finite; then necessary
and sufficient for a~semigroup~$S$ to be contained in $\mathbf{V}\malc\mathbf
{W}$ is that the former condition holds for the canonical relational
morphism $\phi:S\to F_\mathbf{W}(A)$ (with respect to $A$) where $A$ is \textbf{some}
generating set of~$S$ and $F_\mathbf{W}(A)$ is the free
{semigroup} in~$\mathbf{W}$ generated by $A$. This allows to
give quite transparent descriptions of the members of the Ma\softl cev
products of the forms $\mathbf{V}\malc\mathbf{RZ}$, $\mathbf{V}
\malc\mathbf{LZ}$ and $\mathbf{V}\malc\mathbf{RB}$, for an arbitrary pseudovariety $\mathbf{V}$.

\begin{Lemma}
\label{malcrz}

A semigroup $S$ belongs to $\mathbf{V}
\malc\mathbf{RZ}$ if and only if each principal left ideal of~$S$
belongs to~$\mathbf{V}$; dually, $S$ belongs to $\mathbf{V}\malc
\mathbf{LZ}$ if and only if each principal right ideal of~$S$ belongs
to~$\mathbf{V}$.
\end{Lemma}

\begin{proof}
Let $S$ be a semigroup and let  $RZ(S)$ be the (free) right zero
semigroup  on the set~$S$. {Thus, $RZ(S)$ is just the set $S$
endowed with right zero multiplication.} Let $\phi$ be the canonical
relational morphism $S\to RZ(S)$ (with respect to  $S$). Then, for each  $a\in RZ(S)$ ($=S$) we have
${a}\phi^{-1}=L({a})$ whence, according to the  {discussion}
before the statement of the Lemma, $S\in\mathbf{V}\malc\mathbf{RZ}$
if and only if $L({a})\in\mathbf{V}$ for each ${a}\in S$.
The dual case is proved analogously.
\end{proof}

\begin{Lemma}
\label{malcrb}
A semigroup $S$
 belongs to  $\mathbf{V}\malc\mathbf{RB}$ if
and only if, for all $a,b\in S$, the subsemigroups $\{ab\}\cup aSb$
 and  $\{a,a^2\}\cup aSa$ belong to~$\mathbf{V}$.
\end{Lemma}

\begin{proof} Let  $S$ be a semigroup and
let $RB(S)$ be the free
rectangular band on the set~$S$. That is,  $RB(S)$ is the set
$S\times S$ endowed with multiplication
$(a,b){\cdot}(c,d)=(a,d)$. Let further $\phi:S\to RB(S)$ be the
canonical relational morphism with respect to  $S$;
that is,   $\phi$ is the subsemigroup of ${S\times RB(S)}$ generated
by all pairs $(r,(r,r))$, $r\in S$. Then  $S\in\mathbf{V}\malc
\mathbf{RB}$ if and only if $(a,b)\phi^{-1}\in\mathbf{V}$ for all
$a,b\in S$. Now it remains to observe that  $(a,b)
\phi^{-1}=\{ab\}\cup aSb$ for $a\neq b$ and $(a,a)\phi^{-1}=\{a,a^2\}\cup aSa$.
\end{proof}

Lemma \ref{malcrz}  and Lemma \ref{malcrb} are very special instances of the `basis theorem' for Ma\softl cev products by Pin and Weil, see theorem 4.1 in \cite{pinweil}. As mentioned in the introduction, Ka\softd ourek considered in \cite{kadpcs} semigroups $S$ all of whose subsemigroups of the form $aSb$ with $a,b\in S$ are block groups; he called such semigroups \emph{aggregates of block groups}. Whether or not a semigroup is a block group depends entirely on its set of idempotents. It is further clear that there are no idempotents in  $(\{ab\}\cup aSb)\setminus aSb$ nor in $(\{a,a^2\}\cup aSa)\setminus aSa$. It follows that $\{ab\}\cup aSb$ is a block group if and only if so is $aSb$ and likewise $\{a,a^2\}\cup aSa$ is a block group if and only if so is $aSa$, for arbitrary $a,b\in S$. Consequently, the pseudovariety $\mathbf{BG}\malc\mathbf{RB}$ is comprised exactly  of all aggregates of block groups. This observation is implicitly contained in \cite{kadpcs} but (unfortunately) is not mentioned explicitly. In any case, we may formulate a criterion for membership in $\mathbf{BG}\malc\mathbf{RB}$.

\begin{Cor}\label{agbg} A semigroup $S$ belongs to  $\mathbf{BG}\malc\mathbf{RB}$ if and only if, for each choice of elements $a,b\in S$, the semigroup $aSb$ is a block group.
\end{Cor}
Next we present  a similar description of the members of the
semidirect product pseudovariety  $\mathbf{V}*\mathbf{RZ}$, but only for the case when $\mathbf{V}$ is local. Here we
use the derived semigroupoid theorem \cite{tilson, qtheory}.

\begin{Prop}
\label{astrz}     

Let $\mathbf{V}$ be a~local pseudovariety;  a semigroup $S$ belongs to $\mathbf
{V}*\mathbf{RZ}$ if and only if the image $L({a})^\rho$ of each principal left ideal $L({a})$ of $S$  under the right regular
representation $\rho^{L(a)}$ belongs to~$\mathbf{V}$.
\end{Prop}

\begin{proof} Let $S$ be a semigroup, let $RZ(S)$
be the (free) right zero semigroup on the set~$S$ and let
$\phi:S\to RZ(S)$ be the canonical relational morphism (with
respect to $S$). Then, according to the
derived semigroupoid theorem, the semigroup~$S$
belongs to $\mathbf{V}*\mathbf{RZ}$ if and only if the derived
semigroupoid $D_\phi$ of $\phi$ divides a~member of~$\mathbf{V}$. (For a proof that the consideration of the canonical relational morphism $S\to RZ(S)$ suffices here,  see  Proposition 3.5 in combination with Theorem 3.3 in \cite{almeidaweil}.)
 Since $\mathbf{V}$ is local, this happens
if (and only if) each local semigroup of $D_\phi$ is in~$\mathbf{V}$.
The set of objects of the {derived} semigroupoid $D_\phi$
is $RZ(S)\cup\{1\}\ =S\cup\{1\}$  where $1$
represents a new identity adjoined to $S$. However, the local
semigroup at the object~$1$ is  empty. So
consider any other object~${a}$, say. Then, as in the proof of Lemma \ref{malcrz}, we
have ${a}\phi^{-1}=L({a})$, the principal
left ideal of~$S$ generated by~$a$. According to the construction of the derived
semigroupoid $D_\phi$, the local semigroup at
the given object~${a}$ may be identified with the semigroup of all mappings
$L({a})\to L({a})$ induced by multiplication
on the right by elements of  $L(a)$ which semigroup by definition coincides with $L({a})^\rho$.
\end{proof}
Finally, we recall the well-known equalities
$$\mathbf{CS}=\mathbf{G}\malc\mathbf{RB}\mbox{ and }\mathbf{CS}=\mathbf
{G}*\mathbf{RZ}.$$ The former equality is a consequence of Proposition \ref{malcrb} while the latter
is a consequence of Proposition \ref{astrz}.

\subsection{Pseudoidentity bases  of pseudovarieties of the
form $\mathbf{V}*\mathbf{RZ}$}

Proposition \ref{astrz}     
provides us with a~method to obtain a basis of pseudoidentities of $\mathbf{V}*\mathbf{RZ}$ from a basis of~$\mathbf{V}$ if $\mathbf
{V}$ is a local pseudovariety.

\begin{Cor}
\label{basis*rz}

Let  {$\mathbf{V}$ be a local pseudovariety and let $\Sigma$ be a basis of pseudoidentities of $\mathbf{V}$.
{Then} a basis of $\mathbf{V}*\mathbf{RZ}$ is given by the set of  }all pseudoidentities of the form
     \[x\pi(y_1x,\dots,y_nx)
     =x\sigma(y_1x,\dots,y_nx)\]
where the pseudoidentity $\pi(x_1,\dots,x_n)=\sigma(x_1,\dots,x_n)$ is a~member of\,\
$\Sigma$ and, for each~$i$, $y_i$ can take the value $x_i$ or the empty
value, and $x$ is a~variable not {contained} in $\{x_1,\dots,x_n\}$.
\end{Cor}

\begin{proof} Suppose that the semigroup $S$ belongs to
 $\mathbf{V}*\mathbf{RZ}$ and let $\pi=\sigma$ be a pseudoidentity from $\Sigma$. Choose any ${a}\in S$
and any ${b}_1,\dots,{b}_n\in S\cup\{1\}$ where $1$ represents
a~new identity {adjoined} to~$S$. Then ${b}_1{a},\dots,
{b}_n{a}$ and therefore also $\pi({b}_1{a},\dots,
{b}_n{a})$ and $\sigma({b}_1{a},\dots,{b}_n
{a})$ are well defined elements of {the principal left ideal}
$L({a})$. Since $L({a})^\rho$ belongs to~$\mathbf{V}$
{by Proposition~\ref{astrz}} and since~$\mathbf{V}$
satisfies the pseudoidentity $\pi(x_1,\dots,x_n)=\sigma(x_1,\dots,x_n)$,
{substituting $\rho_{b_ia}$ for $x_i$, for each~$i$,} we obtain
    \[\pi(\rho_{{b}_1{a}},\dots,\rho_{{b}_n{a}})=
    \sigma(\rho_{{b}_1{a}},\dots,\rho_{{b}_n{a}})\]
and therefore also
    \[\rho_{\pi({b}_1{a},\dots,{b}_n{a})}=
    \rho_{\sigma({b}_1{a},\dots,{b}_n{a})}.\]
 This means that multiplication by
$\pi({b}_1{a},\dots,{b}_n{a})$ and  {by} $\sigma
({b}_1{a},\dots,{b}_n{a})$ on the right induces
the same transformation $L({a})\to L({a})$.
In particular,
the elements $a\pi(b_1a,\dots,b_na)$ and $a\sigma(b_1a,\dots,b_na)$ coincide.
Since $a\in S$ and $b_1,\dots b_n\in S\cup\{1\}$ have been arbitrarily chosen, $S$ satisfies each  pseudoidentity of the (above-mentioned) form
     \[x{\pi(y_1x,\dots,y_nx)}=x\sigma(y_1x,\dots,y_nx).\]

Let conversely $S$ be a~semigroup which satisfies all
{pseudo}identities of the form $x\pi(y_1x,\dots,y_nx)
=x\sigma(y_1x,\dots,y_nx)$ {described above}. Then,
reading the preceding paragraph backward, we see that the image
$L({a})^\rho$ of each principal left ideal $L(a)$ of $S$ under $\rho^{L(a)}$
satis\-fies all pseudoidentities
from $\Sigma$. This means that
$L({a})^\rho$ belongs to $\mathbf{V}$ and so $S$ belongs  to $\mathbf{V}*\mathbf{RZ}$ by
Proposition~\ref{astrz}.
\end{proof}

Corollary \ref{basis*rz} is known and has been formulated as Theorem 2.2 in
\cite{auintrotter};  the well-informed reader will recognize  it as a very special case (of a
proved instance) of the Almeida--Weil basis theorem  \cite[Cor.
5.4]{almeidaweil}. For more information on the basis theorem we refer the interested reader to Section~7
in~\cite{almeidaprof} and to  Section 3.7
in~\cite{qtheory}.
\section{The main result}
\label{mainresult}

\subsection{Semantic characterization of the members
of $\mathbf{PCS}$}
The main result of the paper  can be
formulated as follows.

\begin{Thm}
\label{pcsisequalto}

The following equalities of pseudovarieties hold:
     \[\mathbf{PCS}=\mathbf{J}\malc\mathbf{CS}=
     \mathbf{BG}\malc\mathbf{RB}=\mathbf{ER}\malc\mathbf{RZ}\cap
     \mathbf{EL}\malc\mathbf{LZ}=\mathbf{BG}*\mathbf{RZ}.\]
\end{Thm}

\begin{proof} We shall verify the inclusions
\[\mathbf{PCS}\subseteq\mathbf{J}\malc\mathbf{CS}\subseteq
     \mathbf{BG}\malc\mathbf{RB}\subseteq\mathbf{ER}\malc\mathbf{RZ}\cap
     \mathbf{EL}\malc\mathbf{LZ}\subseteq\mathbf{BG}*\mathbf{RZ}\subseteq
     \mathbf{PCS}.\]

1. The inclusion $\mathbf{PCS}\subseteq\mathbf{J}\malc\mathbf{CS}$ has been proved by Ka\softd ourek \cite[Theorem 4.2]{kadpcs}. Our proof uses the same relational morphism but otherwise is different and shorter.
Since $\mathbf{P'CS}=\mathbf{PCS}$ it is
sufficient to show that $\mathfrak{P}(S)$  {belongs to} $\mathbf{J}\malc
\mathbf{CS}$  for each $S\in\mathbf{CS}$.

So, let $S$ be a completely simple semigroup. Then the {quotient}
set $S/\mathrsfs{R}$ forms a~left zero semigroup under set-wise
multiplication and the {quotient} mapping $R:S\to S/\mathrsfs
{R}$, ${a}\mapsto R_{{a}}$ is a~homomorphism.
{Likewise, the quotient set} $S/\mathrsfs{L}$ is a~right zero
semigroup and the {quotient} mapping $L:S\to S/\mathrsfs{L}$,
${a}\mapsto L_{{a}}$ is a~homomorphism. Denote by
{ {$\mathfrak{R}$}}
and  {$\mathfrak{L}$} the respective  {induced {homomorphisms}}\,\
$ {{\mathfrak{R}}}:\mathfrak{P}(S)\to\mathfrak{P}(S/\mathrsfs{R})$ and
$ {\mathfrak{L}}:\mathfrak{P}(S)\to\mathfrak{P}(S/\mathrsfs{L})$. Note that
$\mathfrak{P}(S/\mathrsfs{R})$ is  a left zero semigroup
 while
$\mathfrak{P}(S/\mathrsfs{L})$ is  a right zero semigroup.
Moreover, for every $X\in\mathfrak{P}(S)$  {we have}
     \[X {\mathfrak{R}}=\{{U}\in S/\mathrsfs{R}\mid X\cap{U}\neq
     \emptyset\}\mbox{\ \ and\ \ }X {\mathfrak{L}}=\{{V}\in
    S/\mathrsfs{L}\mid X\cap{V}\neq\emptyset\}.\]
Consider the direct product
     \[T:=\mathfrak{P}(S/\mathrsfs{R})\times S\times\mathfrak{P}
     (S/\mathrsfs{L})\]
 {which} is again a~completely simple semigroup. For
{every} $X\in\mathfrak{P}(S)$ set
     \[X\Phi:=\{(X {\mathfrak{R}},x,X {\mathfrak{L}})\mid x\in X\}
     =\{X {\mathfrak{R}}\}\times X\times\{X {\mathfrak{L}}\}.\]
Then $\Phi$ is  a relational morphism $\mathfrak{P}
(S)\to T$.
Let $(\mathcal{A},e,\mathcal{B})$ be an idempotent of $T$ with $(\mathcal {A},e,\mathcal{B})\Phi^{-1}$ non-empty; note that $e$ is an idempotent of $S$.  We intend to show that the semigroup $(\mathcal {A},e,\mathcal{B})\Phi^{-1}$ is $\mathrsfs{J}$-trivial. In order to do so, it is sufficient to verify that   the set $\mathrm{Reg}((\mathcal {A},e,\mathcal{B})\Phi^{-1})$ of all  regular elements is contained in a $\mathrsfs{J}$-trivial subsemigroup of $(\mathcal {A},e,\mathcal{B})\Phi^{-1}$. Then each regular $\mathrsfs{J}$-class of $(\mathcal {A},e,\mathcal{B})\Phi^{-1}$ is trivial whence the semigroup $(\mathcal {A},e,\mathcal{B})\Phi^{-1}$ is itself  $\mathrsfs{J}$-trivial.

 Set
$$S(\mathcal{A},\mathcal{B}):=\bigcup_{U\in \mathcal{A},V\in \mathcal{B}}U\cap V$$ and let $I$ be the (completely simple) subsemigroup generated by all idempotents of $S(\mathcal{A},\mathcal{B})$. Then $I$ is an  idempotent  of $(\mathcal {A},e,\mathcal{B})\Phi^{-1}$. Let
$$\mathfrak{M}_I:=\{X\in (\mathcal {A},e,\mathcal{B})\Phi^{-1}\mid IX=XI=X\}$$ be the local submonoid of $(\mathcal {A},e,\mathcal{B})\Phi^{-1}$ with respect to $I$. We claim that $X\supseteq I$ for each $X\in \mathfrak{M}_I$. Indeed, let $X\in \mathfrak{M}_I$ and, for an arbitrary idempotent $g$ of $I$ let $I_g$ be the $\mathrsfs{H}$-class of $I$ containing $g$. First observe that $I_e=eI_e\subseteq XI=X$ since $e\in X$. Next, let $f$ be the unique idempotent such that $e\mathrel{\mathrsfs{R}}f\mathrel{\mathrsfs{L}}g$; then $I_g=gI_ef\subseteq IXI=X$. So, $X$ contains each $\mathrsfs{H}$-class of $I$ and thus contains $I$ itself. Therefore, the monoid $\mathfrak{M}_I$ is ordered by reverse inclusion $\supseteq$ with $I$ the greatest element, and so $\mathfrak{M}_I$ is $\mathrsfs{J}$-trivial \cite{pinorder}. It suffices  now to show that $\mathrm{Reg}((\mathcal {A},e,\mathcal{B})\Phi^{-1})$ is contained in $\mathfrak{M}_I$.

For each $X\subseteq S$ denote by $E_X$ the set of all idempotents of $X$.
Let  $F$ be an idempotent of $(\mathcal {A},e,\mathcal{B})\Phi^{-1}$; then $F$ is a (completely simple) subsemigroup of $S(\mathcal{A},\mathcal{B})$ with
$F\mathfrak{R}=\mathcal{A}$ and $F\mathfrak{L}=\mathcal{B}$. Let $U\in\mathcal{A}$, $V\in \mathcal{B}$, $a\in F\cap U$, $b\in F\cap V$; then $(ab)^\om\in F\cap U\cap V$. It follows that $F$ contains all idempotents of $S(\mathcal{A},\mathcal{B})$, hence $F\supseteq I$ and $E_F=E_I$. Then
$F=F^2\supseteq FI\supseteq FE_I=FE_F=F$ whence $FI=F$ and dually also $IF=F$. Moreover, for each  regular element $X$ of $(\mathcal {A},e,\mathcal{B})\Phi^{-1}$ there exist idempotents $F,G$ such that $FX=X=XG$; then $IX=IFX=FX=X$ and dually also $XI=X$. Consequently, $\mathrm{Reg}((\mathcal{A},e,\mathcal{B})\Phi^{-1})$ is contained in $\mathfrak{M}_I$, as required.

2. The inclusion $\mathbf{J}\malc\mathbf{CS}\subseteq
\mathbf{BG}\malc\mathbf{RB}$ is a consequence of the
equality $\mathbf{BG}=\mathbf{J}\malc\mathbf{G}$ (see Section 7 in \cite{pinsuccess})  and of  standard facts
{concerning} the operation $\malc$ {(see Lemma 1.4
in~\cite{weildotdepth})}:
     \[\mathbf{J}\malc\mathbf{CS}
     =\mathbf{J}\malc(\mathbf{G}\malc\mathbf{RB})
     \subseteq(\mathbf{J}\malc\mathbf{G})\malc\mathbf{RB}
     =\mathbf{BG}\malc\mathbf{RB}.\]

3. An indirect argument will be used to prove the
inclusion
     \[\mathbf{BG}\malc\mathbf{RB}\subseteq
     \mathbf{ER}\malc\mathbf{RZ}\cap\mathbf{EL}\malc\mathbf{LZ}.\]
Indeed, suppose that $S\notin\mathbf{ER}\malc\mathbf{RZ}$. Then,
according to Lemma~\ref{malcrz}, there exists a principal left ideal $L({a})$
of $S$ which does not belong to $\mathbf{ER}$.
Consequently there exist distinct idempotents $e,f$ in $L(a)$
forming a (non-trivial) right zero subsemigroup.  This right zero
subsemigroup  then is, in fact, contained in the left
ideal $Sa$. However, since $e\{e,f\}=\{e,f\}$, the
right zero subsemigroup $\{e,f\}$ is even
contained in  $eSa$, so that the latter subsemigroup of $S$ cannot be a block group.  Corollary \ref{agbg}
now implies that $S\notin\mathbf{BG}\malc\mathbf{RB}$. By the dual argument,
one {also} shows that $S\notin\mathbf{EL}\malc\mathbf{LZ}$
{entails} {that} $S\notin\mathbf{BG}\malc\mathbf{RB}$.

4. In order to prove the inclusion
     \[\mathbf{ER}\malc\mathbf{RZ}\cap\mathbf{EL}\malc\mathbf{LZ}
     \subseteq \mathbf{BG}*\mathbf{RZ},\]
 {suppose} that $S\in\mathbf{ER}\malc\mathbf{RZ}\cap\mathbf{EL}
\malc\mathbf{LZ}$. According to Lemma~\ref{malcrz} this means that
each principal left ideal of~$S$ is in $\mathbf{ER}$ and each
principal right ideal of~$S$ is in $\mathbf{EL}$. It is well known and easy to see that $\mathbf{BG}$ is local (by taking advantage of the \emph{consolidation operation}  \cite{tilson, qtheory} --- the reader may also  consult the proof of Proposition 1  in \cite{steinbpcs}). Therefore we may now apply Proposition~\ref{astrz}: in order to show that $S$ belongs to $\mathbf{BG}*\mathbf{RZ}$,  it is sufficient to verify
that the image
$L({a})^\rho$ of each principal left ideal $L({a})$ of $S$  under $\rho^{L(a)}$ belongs to
$\mathbf{ER}\cap\mathbf{EL}$ ($=\mathbf{BG}$). Since each principal
left ideal $L({a})$ is in $\mathbf{ER}$ so is its image
$L({a})^\rho$. It remains to prove that $L({a})^\rho$ is in $\mathbf{EL}$. We need to show that each (at most) two-element left zero
subsemigroup $\{\rho_{ca},\rho_{da}\}$ of $L({a})^\rho$ is trivial; here $ca,da$ are arbitrary elements of $L(a)$ (${c}$ and/or ${d}$
may be the empty symbol). By assumption, the equalities
     \[{a}({ca})={a}({ca})^2,\
     {a}({da})={a}({da})^2,\
     {a}({ca})({da})={a}({ca}),\
     {a}({da})({ca})={a}({da})\]
hold. Shifting the brackets to the left we
observe that the inner left translations $\lambda_{{ac}},\lambda_{{ad}}:
R({a})\to R({a})$   form a left zero subsemigroup of ${}^\lambda\!{R(a)}$, the
image of $R({a})$ under $\lambda^{R(a)}$.  Since
$R({a})$ is in $\mathbf{EL}$ so is {${}^\lambda\!{R(a)}$}. Hence $\{\lambda_{ac},\lambda_{ad}\}$ is a trivial semigroup and so $\lambda_{ac}=\lambda_{ad}$. It follows that $(ac)a=(ad)a$  and
therefore also $a(ca)=a(da)$ hold. But then
$\rho_{{ca}}=\rho_{{da}}$, that is, the left zero semigroup $\{\rho_{ca},\rho_{da}\}$ is  trivial.

5. Finally, in order to verify the inclusion
$\mathbf{BG}*\mathbf{RZ}\subseteq\mathbf{PCS}$, we may argue exactly as in \cite{steinbpcs}, namely we may use the inclusion $\mathbf{PG}*\mathbf{RZ}\subseteq\mathbf{P}(\mathbf{G}*\mathbf{RZ})$
 (proved  in~\cite{auinsteinbdivisions}) and  the equalities $\mathbf{PG}=\mathbf{BG}$ and $\mathbf{CS}=\mathbf{G}*\mathbf{RZ}$.
 \end{proof}

Just as in
{Theorem~4.2 of}~\cite{kadpcs}, the first step in the proof
of Theorem~\ref{pcsisequalto} actually proves the inclusion
$\mathbf{PQ}\subseteq\mathbf{J}\malc\mathbf{Q}$ for \textbf{each}
pseudovariety $\mathbf{Q}$ of completely simple semigroups.
The reason for this is that the left zero
semigroup $\mathfrak{P}(S/\mathrsfs{R})$ is trivial if and only if so is $S/\mathrsfs{R}$ and the
right zero semigroup $\mathfrak{P}(S/\mathrsfs{L})$ is trivial if and only if so is $S/\mathrsfs{L}$. Consequently, both power semigroups belong to the
pseudovariety generated {by}~$S$,
and therefore {also}~$T$ belongs
to the pseudovariety generated {by}~$S$.

As already mentioned,  two papers  are  devoted to the study of the pseudovariety $\mathbf{PCS}$, namely Steinberg's paper \cite{steinbpcs} and Ka\softd ourek's paper \cite{kadpcs}. The main result in the former is  the inclusion $\mathbf{PCS}\subseteq \mathbf{BG}*\mathbf{RZ}$ while the main results of the latter are the inclusion $\mathbf{PCS}\subseteq \mathbf{J}\malc \mathbf{CS}$ and the equality $\mathbf{BG}\malc\mathbf{RB}=\mathbf{J}\malc\mathbf{CS}$. Whereas the inclusion of the latter paper forms an essential ingredient of our proof we get the inclusion of the former  and the equality of the latter  for free.
 It should be noted however, that we do make use of the celebrated equality
$\mathbf{PG}=\mathbf{BG}$.

The last three items in the chain of equalities in Theorem~\ref{pcsisequalto} give rise to three (equivalent)  structural characterizations of the members of $\mathbf{PCS}$. These characterizations in turn lead to transparent algorithms for testing membership in $\mathbf{PCS}$. Indeed, immediately from Corollary \ref{agbg} we get: a semigroup $S$ belongs to  $\mathbf{PCS}$ if and only if $aSb$ is a block group for all $a,b\in S$ that is, if and only if it is an aggregate of block groups in the sense of Ka\softd ourek \cite{kadpcs}. Next, expressing membership in $\mathbf{ER}\malc\mathbf{RZ}\cap \mathbf{EL}\malc\mathbf{LZ}$ in terms of Lemma \ref{malcrz} we  obtain: a semigroup $S$ belongs to  $\mathbf{PCS}$ if and only if each principal left ideal of $S$ is in $\mathbf{ER}$ and each principal right ideal of $S$ is in $\mathbf{EL}$. Finally, applying Proposition \ref{astrz} to $\mathbf{BG}*\mathbf{RZ}$ we see: a semigroup $S$ belongs to $\mathbf{PCS}$ if and only if the image $L(a)^\rho$ of each principal left ideal $L(a)$ of $S$ under the right regular representation $\rho^{L(a)}$ is a block group. In the latter two cases, the one-sided ideals $L(a)$ and $R(a)$ may be replaced with the one-sided ideals $Sa$ and $aS$, respectively.

The equality $\mathbf{PG}=\mathbf{BG}$ has two analogues in the completely simple case, namely, on the one hand Steinberg's equality
 $$\mathbf{PCS}=\mathbf{P}(\mathbf{G}*\mathbf{RZ})=\mathbf{BG}*\mathbf{RZ},$$ on the other hand the equality
 $$\mathbf{PCS}=\mathbf{P}(\mathbf{G}\malc\mathbf{RB})=\mathbf{BG}\malc\mathbf{RB}.$$
Finally, both equalities $\mathbf{PG}=\mathbf{J}*\mathbf{G}$ and $\mathbf{PG}=\mathbf{J}\malc\mathbf{G}$ (see \cite{pinsuccess}) have their obvious analogues, namely
$$\mathbf{PCS}=\mathbf{J}*\mathbf{CS}\mbox{ and }\mathbf{PCS}=\mathbf{J}\malc\mathbf{CS}.$$
The latter occurs in the statement of Theorem \ref{pcsisequalto} and is important  in its proof. In contrast, the author is not aware of a direct argument for the former: it follows merely from the equalities $\mathbf{PCS}=\mathbf{BG}*\mathbf{RZ}$, $\mathbf{BG}=\mathbf{J}*\mathbf{G}$, $\mathbf{CS}=\mathbf{G}*\mathbf{RZ}$ and the associativity of the operation $*$, see \cite{kadpcs}.

\subsection{Syntactic characterization of the members of
$\mathbf{PCS}$}

The last three items in the chain of equalities  in Theorem~\ref{pcsisequalto} also give rise to three
syntactic characterizations of the pseudovariety $\mathbf{PCS}$.
First, using the fact that  $\mathbf{BG}$ is
defined by the pseudoidentity $(x^\om y^\om)^\om=(y^\om x^\om)^\om$,
it follows from   Corollary \ref{agbg} that  \mbox{$\mathbf{BG}\malc
\mathbf{RB}$} is defined by the single pseudoidentity
     \[((sxt)^\om(syt)^\om)^\om=((syt)^\om(sxt)^\om)^\om,\]
see \cite{kadpcs}. This statement can also be deduced from
the already mentioned Pin--Weil basis theorem for Ma\softl cev products \cite{pinweil}. Further on, as mentioned in the preliminaries,
the pseudoidentity
     $(y^\om z^\om)^\om =(y^\om z^\om)^\om y^\om$ defines $\mathbf{ER}$.
Moreover, every idempotent in the principal left ideal $L(x)$
of a semigroup~$S$  can be written
as $(ax)^\om$ for some element~$a$ of $S$. {In view of Lemma~\ref
{malcrz}, it} follows that the pseudovariety $\mathbf{ER}\malc
\mathbf{RZ}$ is defined by the pseudoidentity
     \begin{equation}
     \label{errz}
     ((ax)^\om(bx)^\om)^\om
     =((ax)^\om(bx)^\om)^\om(ax)^\om.\tag{$\dag$}
     \end{equation}
Dually,   $\mathbf{EL}\malc\mathbf{LZ}$ is
defined by the pseudoidentity
     \begin{equation}
     \label{ellz}
     ((xb)^\om(xa)^\om)^\om
     =(xa)^\om((xb)^\om(xa)^\om)^\om.\tag{$\ddag$}
     \end{equation}
(Again, both statements are immediate consequences of the above-mentioned
basis theorem for Ma\softl cev products  \cite{pinweil}.) Consequently, the two pseudoidentities
(\ref{errz},\ref{ellz}) form  a basis of
     $\mathbf{ER}\malc\mathbf{RZ}\cap\mathbf{EL}\malc\mathbf{LZ}$.
Finally, as a consequence of Corollary \ref{basis*rz} we
obtain a basis of pseudoidentities of $\mathbf{BG}*\mathbf{RZ}$ as
follows.

\begin{Cor}
\label{basisBG*RZ}
The pseudovariety $\mathbf{BG}*\mathbf{RZ}$ is defined by the pseudoidentity
     \begin{equation} x((yx)^\omega(zx)^\omega)^\omega
     =x((zx)^\omega(yx)^\omega)^\omega.\tag{$\star$}\end{equation}
\end{Cor}
\begin{proof} It has been noticed in the preliminaries that the single pseudoidentity
     $(y^\omega z^\omega)^\omega=(z^\omega y^\omega)^\omega$ forms a basis of $\mathbf{BG}$.
As already mentioned in part 4 of the proof of Theorem \ref{pcsisequalto}, the pseudovariety $\mathbf{BG}$ is local.
Thus
we may apply to
the pseudovariety $\mathbf{BG}*\mathbf{RZ}$ the rule formulated in Corollary~\ref{basis*rz}. In this way, we obtain
four pseudoidentities, {each} of the form
 $$x((ax)^\omega(bx)^\omega)^\omega=x((bx)^\omega(ax)^\omega)^\omega$$
where $a$ as well as $b$ can be a variable or the empty symbol.
If $a$, say, represents the empty symbol then $(ax)^\omega$ becomes
$x^\omega$ which is the same as $(xx)^\omega$. So, in the case
{that} $a$ and/or $b$ represent the empty symbol, the corresponding
pseudoidentity can be obtained from the one in which $a$ as well as $b$
represent variables, simply by substituting the variable $a$ and/or $b$
with the variable $x$. Therefore the pseudoidentity ($\star$) defines
$\mathbf{BG}*\mathbf{RZ}$.
\end{proof}
We arrive at the main theorem in this section.
\begin{Thm}
\label{identitybasis}
Each of the  pseudoidentity systems
$\mathrm{(i)}$ -- $\mathrm{(iii)}$ forms a basis of pseudoidentities of  $\mathbf{PCS}$:
\begin{enumerate}
\item[(i)] $((sxt)^\om(syt)^\om)^\om=((syt)^\om(sxt)^\om)^\om$,
\item[(ii)] $((ax)^\om(bx)^\om)^\om=((ax)^\om(bx)^\om)^\om(ax)^\om$
            and\\
            $((xb)^\om(xa)^\om)^\om=(xa)^\om((xb)^\om(xa)^\om)^\om$,
\item[(iii)] $x((ax)^\om(bx)^\om)^\om=x((bx)^\om(ax)^\om)^\om.$
\end{enumerate}
\end{Thm}
\noindent It is interesting to see directly, that is, using only syntactic
arguments, that the  systems of pseudoidentities (i)--(iii) are equivalent. In the following, we shall present such a syntactic derivation.

In order to prove that (i) implies (ii) we first note
that by substituting $ts$ for $x$ in (i) and  {using}
$(stst)^\om=(st)^\om$ we get
     \begin{equation}
     \label{deletex}
     ((st)^\om(syt)^\om)^\om
     =((syt)^\om(st)^\om)^\om.\tag{{$\mathrm{i'}$}}
     \end{equation}
Now, using the equalities $y^\om(y^\om z^\om)^\om=(y^\om z^\om)^\om=(y^\om z^\om)^\om z^\om$
we obtain
     \begin{align*}
     ((ax)^\om(bx)^\om)^\om&=(((ax)^\om(bx)^\om)^\om)^\om&\\
     &=((ax)^\om((ax)^\om(bx)^\om)^\om)^\om&\\
     &=((ax)^\om(amx)^\om)^\om&
     \mbox{\ for\ }m=x(ax)^{\om-1}(bx)^{\om-1}b\\
     &=((amx)^\om(ax)^\om)^\om&\mbox{\ by\ (\ref{deletex})}\\
     &=(((ax)^\om(bx)^\om)^\om(ax)^\om)^\om&
               \end{align*}
and
     \begin{align*}
    ((ax)^\om(bx)^\om)^\om&=(((ax)^\om(bx)^\om)^\om(ax)^\om)^\om& (\natural)\phantom{.}\\
         &=(((ax)^\om(bx)^\om)^\om(ax)^\om)^\om(ax)^\om&\\
         &=((ax)^\om(bx)^\om)^\om(ax)^\om&\mbox{ by }(\natural).
     \end{align*}
\noindent The second pseudoidentity of (ii) is derived analogously
(by reading (\ref{deletex}) from right to left,  that is, the two sides of that pseudoidentity are swapped with each other).

Let us consider next the pseudoidentites (ii) and denote them {by}
(ii1) and (ii2), respectively. {In order to verify that (ii)
implies (iii), we} multiply (ii1) by~$x$ on the left to obtain:
     \begin{align*}
     x{\cdot}((ax)^\om(bx)^\om)^\om&
     =x{\cdot}((ax)^\om(bx)^\om)^\om(ax)^\om&\mbox{\ by (ii1)}\\
     &=(xa)^\om((xb)^\om(xa)^\om)^\om{\cdot}x&
     \mbox{\ shifting brackets\ }\\
     &=((xb)^\om(xa)^\om)^\om{\cdot}x&\mbox{by (ii2)}\\
     &=x{\cdot}((bx)^\om(ax)^\om)^\om&\mbox{\ shifting brackets.}&
\end{align*}

Finally, we intend to derive (i) from (iii). For this purpose,
we note that shifting the brackets in (iii) we get an equivalent
pseudoidentity, namely:
     \begin{equation}
     \label{shiftiii}
     ((xa)^\om(xb)^\om)^\om x=((xb)^\om(xa)^\om)^\om x.\tag{iii$^s$}
     \end{equation}
Having this in mind and assuming (iii), we obtain
     \begin{align*}
     ((sxt)^\om(syt)^\om)^\om&
     =((sxt)^\om(syt)^\om)^\om((sxt)^\om(syt)^\om)^\om&\\
     &=l\underline{t}((sx\underline{t})^\om(sy\underline{t})^\om)^\om&\\
     &=l\underline{t}((sy\underline{t})^\om(sx\underline{t})^\om)^\om&
     \mbox{\ by (iii)}\\
     &=((sxt)^\om(syt)^\om)^\om((syt)^\om(sxt)^\om)^\om&\\
     &=((\underline{s}xt)^\om(\underline{s}yt)^\om)^\om\underline{s}r&\\
     &=((\underline{s}yt)^\om(\underline{s}xt)^\om)^\om \underline{s}r&
     \mbox{\ by\ (\ref{shiftiii})}\\
     &=((syt)^\om(sxt)^\om)^\om((syt)^\om(sxt)^\om)^\om&\\
     &=((syt)^\om(sxt)^\om)^\om,&
\end{align*}
where
     \[l=((sxt)^\om(syt)^\om)^{\om-1}(sxt)^\om(syt)^{\om-1}sy\]
      and
      \[r=yt(syt)^{\om-1}(sxt)^\om((syt)^\om(sxt)^\om)^{\om-1},\]
which proves that (iii) implies (i). Altogether we have given an alternative,
purely syntactic proof of the  {equalities}
     \[\mathbf{BG}\malc\mathbf{RB}
     =\mathbf{ER}\malc\mathbf{RZ}\cap\mathbf{EL}\malc\mathbf{LZ}
     =\mathbf{BG}*\mathbf{RZ}.\]
{What is more, we actually have proved the implications:
     \[\mathrm{(i)}\Rightarrow\mathrm{(i')}\Rightarrow\mathrm{(ii)}
     \Rightarrow\mathrm{(iii)}\Rightarrow\mathrm{(i)}.\]
{}From this it follows that $\mathrm{(i')}$ is yet another pseudoidentity
defining $\mathbf{PCS}$. Moreover, each of the {pseudo}identity bases of Theorem \ref{identitybasis}
provides a new algorithm for deciding membership in $\mathbf{PCS}$.

\subsection{Another algorithmic problem}\label{constructivedivision} For two completely simple semigroups $B$ and $C$, the direct product $\mathfrak{P}(B)\times \mathfrak{P}(C)$ embeds in $\mathfrak{P}(B\times C)$. Hence the pseudovariety $\mathbf{PCS}$ consists of all divisors of the power semigroups $\mathfrak{P}(C)$ of all completely simple semigroups $C$. For a given semigroup $S\in \mathbf{PCS}$ (that is, a semigroup $S$ that has successfully passed a membership test for $\mathbf{PCS}$),  one may ask how to construct a  completely simple semigroup $C$ for which $S$ divides $\mathfrak{P}(C)$. The analogous question for  the members of $\mathbf{PG}$ had been a standing problem and was solved by Steinberg and the author \cite{auinsteinbdivisions}. Based on that  we shall construct a solution for the present case.

So, let $S\in\mathbf{PCS}$, let $RZ(S)$ be the (free) right zero semigroup on $S$ and let $\phi:S\to RZ(S)$ be the canonical relational morphism  as in Lemma \ref{malcrz}. The derived semigroupoid $D_\phi$ of $\phi$ has set of objects $RZ(S)\cup\{1\}=S\cup\{1\}$ where $1\notin S$. The set $D_\phi(a,1)$ of all morphisms $a\to 1$ is empty for each object $a$; for $b\in S$, the set $D_\phi(1,b)$ can be identified with the set of all triples $(1,\gamma,b)$ where $\gamma$ is a `translation' $\{1\}\to L(b)$, $1\mapsto 1g=g$ with $g\in L(b)$; finally,
for two objects $a,b\in S$, the set $D_\phi(a,b)$  can be identified with the set of all triples $(a,\sigma,b)$ where $\sigma$ denotes the mapping $ L(a)\to L(b)$, $x\mapsto xs$ for $s\in L(b)$. The composition of two composable morphisms  is the obvious one:
\begin{equation}\label{product}(a,\sigma,b)(b,\tau,c)=(a,\sigma\tau,c).\tag{$*$}\end{equation}
Let $BG(S)$ be the \emph{consolidated semigroup} of $D_\phi$ as in \cite{tilson, qtheory} (see  also below). Since $S\in \mathbf{BG}*\mathbf{RZ}$, the derived semigroupoid theorem implies that (i) the local semigroups of $D_\phi$ are block groups (and so $BG(S)$ itself is a block group since $BG(S)$ does not contain non-trivial left zero and right zero subsemigroups) and (ii) $S$ divides the wreath product $BG(S)\wr RZ(S)$.

Next, from Theorem 4.5 in \cite{auinsteinbdivisions} we are aware how to construct from $BG(S)$ a group $G(S)$ such that $BG(S)$ divides the power group $\mathfrak{P}(G(S))$. Standard properties of the wreath product then imply that $BG(S)\wr RZ(S)$ and hence also $S$ divide the wreath product $\mathfrak{P}(G(S))\wr RZ(S)$. Finally, in Lemmas 2.1 and 2.2 in \cite{auinsteinbdivisions} a division from $\mathfrak{P}(G(S))\wr RZ(S)$ to  $\mathfrak{P}(G(S)\wr RZ(S))$ is explicitly constructed. Consequently, the original semigroup $S$ divides the power semigroup $\mathfrak{P}(G(S)\wr RZ(S))$ and $G(S)\wr RZ(S)$ is completely simple.

We summarize the procedure as follows: for $S\in \mathbf{PCS}$ let $BG(S)$ be the set consisting of an element $0$ together with all triples $(a,\sigma,b)$ for $a\in S\cup\{1\}$, $b\in S$ and $\sigma$ being a mapping $\tilde L(a)\to L(b)$ of the form $x\mapsto xs$ for some $s\in L(b)$ (here $\tilde L(1)=\{1\}$ while $\tilde L(a)=L(a)$ if $a\ne 1$). Let $BG(S)$ be endowed with the multiplication (\ref{product}) whenever possible and let all other products be defined to be $0$. Then $BG(S)$ is a block group; let $G(S)$ be a group  for which $BG(S)$ divides the power group $\mathfrak{P}(G(S))$, as constructed in \cite{auinsteinbdivisions}. Then $S$ divides the power semigroup $\mathfrak{P}(G(S)\wr RZ(S))$ where $RZ(S)$ is the set $S$ endowed with right zero multiplication, and $G(S)\wr RZ(S)$ is completely simple.
\vskip0.05cm
\noindent
\textbf{Personal remark of the author}. The work on this paper started as a joint project with J.~Ka\softd ourek (initiated by the author). Unfortunately, from the very beginning, this cooperation was overshadowed by dissent on various matters of the joint work. After some while the author felt that there was no longer the possibility to write the paper in a way that would be satisfactory for both authors and  resigned from that cooperation. The outcome is now entirely according to the view of the author, and Ka\softd ourek did not accept co-authorship for it.

\end{document}